\documentclass{amsart}
\usepackage{amssymb}
\numberwithin{equation}{section}

\newcommand{\CAT}{\textup{CAT}}

\theoremstyle{plain}
\newtheorem{theorem}{Theorem}[section]
\newtheorem{lemma}[theorem]{Lemma}

\theoremstyle{definition}
\newtheorem{definition}[theorem]{Definition}

\theoremstyle{remark}

\allowdisplaybreaks
\title[Convergence to a common fixed point of mappings on the unit sphere]{Convergence to a common fixed point of a finite family of nonexpansive mappings on the unit sphere of a Hilbert space}
\author[T.~Ezawa]{Tatsuki~Ezawa}
\address[T.~Ezawa]
{Graduate School of Mathematics, Nagoya University, Chikusa-ku, Nagoya 464-8602, Japan}
\email{m14006q@math.nagoya-u.ac.jp}
\begin{document}
\begin{abstract}
In this paper, we consider the CQ projection method and the shrinking projection method given by a finite family of nonexpansive mappings of the unit sphere of a real Hilbert space, and prove strong convergence theorems to common fixed points of the mappings.
\end{abstract}
\maketitle

\section{Introduction}\label{sec:introduction}

Let $C$ be a nonempty closed convex subset of a real Hilbert space and $T$ a nonexpansive mapping from $C$ onto itself such that the fixed points set $F(T)$ is nonempty. Let $P_C$ be a metric projection to a nonempty closed convex subset. In 2000, Solodov-Svaiter \cite{Solodov-Svaiter} introduced the CQ projection method and in 2003, Nakajo-Takahashi \cite{Nakajo-Takahashi} consider the following iteration:
\begin{align*}
x_1 &:=x\in C, \\
y_n &:=\alpha_nx_n+(1-\alpha_n)Tx_n, \\
C_n &:=\{z\in C\ |\ \|y_n-z\|\leq\|x_n-z\|\}, \\
Q_n &:=\{ z\in C\ |\ \langle x_n-z,x_1-x_n\rangle\geq 0\}, \\
x_{n+1} &:=P_{C_n\cap Q_n}x_1,
\end{align*}
where $\{\alpha_n\}$ satisfies $\lim_{n\to\infty}\alpha_n=0$ and $\sum_{n=1}^\infty|\alpha_{n+1}-\alpha_n|<\infty$ and $P_{C_n\cap Q_n}$ is a metric projection from $C$ onto $C_n\cap Q_n$. Then they showed $\{x_n\}$ is converget to $P_{F(T)}x_1$. In 2006, Nakajo-Shimoji-Takahashi \cite{Nakajo-Shimoji-Takahashi} used $W$-mapping (see Definition \ref{Wmapping}) genereted to consider the followinf iteration:
\begin{align*}
x_1 &:=x\in C, \\
y_n &:=W_nx_n, \\
C_n &:=\{z\in C\ |\ \|y_n-z\|\leq\|x_n-z\|\}, \\
Q_n &:=\{ z\in C\ |\ \langle x_n-z,x_1-x_n\rangle\geq 0\}, \\
x_{n+1} &:=P_{C_n\cap Q_n}x_1,
\end{align*}
and showed $\{x_n\}$ is convergent to $P_Fx_1$, where $F$ is a common fixed points set of generater of $W$-mapping. In 2012, Kimura-Sat\^o \cite{KS3} consider similar iteration in real Hilbert sphere as following:
\begin{align*}
x_1 &:=x\in C, \\
y_n &:=Tx_n, \\
C_n &:=\{z\in C\ |\ d(y_n,z)\leq d(x_n,z)\}, \\
Q_n &:=\{ z\in C\ |\ \cos d(x_1,x_n)\cos d(x_0,x_n)\geq\cos d(x_1,z)\}, \\
x_{n+1} &:=P_{C_n\cap Q_n}x_1
\end{align*}
On the other hand, Takahashi-Takeuchi-Kubota \cite{Takahashi-Takeuchi-Kubota} introduced shrinking projection method in real Hilbert space as following iteration:
\begin{align*}
x_1 &:=x\in C, \\
y_n &:=\alpha_n x_n+(1-\alpha_n)Tx_n, \\
C_n &:=\{z\in C\ |\ \|y_n-z\|\leq \|x_n-z\|\}, \\
x_{n+1} &:=P_{C_n}x_1,
\end{align*}
where $0\leq\alpha_n<a<1$ for all $n\in\mathbb{N}$. Then they showed $\{x_n\}$ is convergent to $P_{(T)}$. In 2009, Kimura-Takahashi \cite{Kimura-Takahashi} considered shrinking projection method in Banach space and in Kimura \cite{Kimura} considered shrinking projection method in a real Hilbert ball which is a example of a Hadamard space.

In this paper, we consider CQ projection mathod and shrinking projection method for a finite family of nonexpansive mappings in a real Hilbert sphere which is a example of a complete $\CAT(1)$ space, that is, we showed following result:

\begin{theorem}[Theorem \ref{CQ}]
Let $C$ be a closed convex subset in real Hilbert shpere $S_H$ such that $d(v,v')\leq\pi/2$ for every $v,v'\in C$. Let $\alpha_{n,1},a_{n,2},\ldots,\alpha_{n,r}$ be real numbers for $n\in\mathbb{N}$ such that $\alpha_{n,i}\in[a,1-a]$ for every $i=1,2,\ldots,r$ where $0<a<1/2$, and let $T_1,T_2,\ldots,T_r$ be a finite number of nonexpansive mappings of C into itself such that $F:=\bigcap_{i=1}^rF(T_i)\neq\emptyset$. Let $W_n$ be the W-mappings of $X$ into itself generated by $T_1,T_2,\ldots,T_r$ and $\alpha_{n,1},\alpha_{n,2},\ldots,\alpha_{n,r}$ for $n\in\mathbb{N}$. For a given point $x_1\in C$, let $\{x_n\}$ be a sequence in $C$ generated by
\begin{align*}
y_n &:=  W_nx_n, \\
C_n &:= \{z\in C\ |\ d(y_n,z)\leq d(x_n,z)\}, \\
Q_n &:= \{ z\in C\ |\ \cos d(x_1, x_n)\cos d(x_n,z)\geq\cos d(x_1,z)\}, \\
x_{n+1} &:= P_{C_n\cap Q_n}x_1
\end{align*}
for all $n\in\mathbb{N}$. Then $\{x_n\}$ is well defined and convergent to $P_Fx_1$.
\end{theorem}

\begin{theorem}[Theorem \ref{Shrinking}]
Let $C$ be a closed convex subset in real Hilbert shpere $S_H$ such that $d(v,v')<\pi/2$ for every $v,v'\in C$. Let $\alpha_{n,1},a_{n,2},\ldots,\alpha_{n,r}$ be real numbers for $n\in\mathbb{N}$ such that $\alpha_{n,i}\in[a,1-a]$ for every $i=1,2,\ldots,r$ where $0<a<1/2$, and let $T_1,T_2,\ldots,T_r$ be a finite number of nonexpansive mappings of C into itself such that $F:=\bigcap_{i=1}^rF(T_i)\neq\emptyset$. Let $W_n$ be the W-mappings of $X$ into itself generated by $T_1,T_2,\ldots,T_r$ and $\alpha_{n,1},\alpha_{n,2},\ldots,\alpha_{n,r}$ for $n\in\mathbb{N}$. For a given point $x_1\in C$, let $\{x_n\}$ be a sequence in $C$ generated by
\begin{align*}
x_1 &:=x\in C, \\
y_n &:=W_nx_n, \\
C_n &:=\{z\in C\ |\ d(y_n,z)\leq d(x_n,z)\}\cap C_{n-1}, \\
x_{n+1} &:=P_{C_n}x_1,
\end{align*}
for all $n\in\mathbb{N}$. Then $\{x_n\}$ is well defined and convergent to $P_Fx_1$.
\end{theorem}

It is essential to show that $\lim_{n\to\infty}d(T_ix_n,x_n)=0$ for every $i=1,2,\ldots,r$ for both theorems. For that we used the Theorem \ref{pal}.

\section{Preliminaries}\label{sec:Preliminaries}

Let $(X,d)$ be a metric space. For $x,y\in X$, a map $c:[0,l]\to X$ is called a geodesic if it satisfies $c(0)=x,c(l)=y$ and $d(c(s),c(t))=|s-t|$ for all $s,t\in [0,l]$. If a geodesic is unique for every $x,y\in X$, we denote the image $c([0,l])$ of $c$ by $[x,y]$ and call it the geodesic segment joining $x$ to $y$. If $[x,y]$ is a unique geodesic segment joining $x$ to $y$ for any $x,y\in X$, then we call $X$ a geodesic metric space.

Let $X$ be a geodesic metric space. For all $x,y,z\in X$, set $\bigtriangleup(x,y,z)=[x,y]\cup[y,z]\cup[z,x]$ and call it a geodesic triangle. For any geodesic triangle $\bigtriangleup(x,y,z)$ satisfying $d(x,y)+d(y,z)+d(z,x)<2\pi$, there exists a spherical triangle $\overline{\bigtriangleup}(\overline{x},\overline{y},\overline{z})$ in $\mathbb{S}^2$ such that each corresponding edge has the same length as that of the original triangle, where $\mathbb{S}^2$ is a unit sphere in a Euclidena space $\mathbb{R}^3$. $\overline{\bigtriangleup}(\overline{x},\overline{y},\overline{z})$ is called the comparison triangle of $\bigtriangleup(x,y,z)$. For $p,q\in\bigtriangleup(x,y,z)$, there exist comparison points $\overline{p},\overline{q}\in \overline{\bigtriangleup}(\overline{x},\overline{y},\overline{z})$. If the inequality
\[
d(p,q)\leq d_{\mathbb{S}^2}(\overline{p},\overline{q})
\]
holds for all $p,q\in \bigtriangleup(x,y,z)$, where $d_{\mathbb{S}^2}$ is the spherical metric on $\mathbb{S}^2$, then we call $X$ a ${\rm CAT(1)}$ space. If $z\in[x,y]$ satisfies $d(y,z)=\alpha d(x,y)$ and $d(x,z)=(1-\alpha)d(x,y)$, we denote $z=\alpha x\oplus(1-\alpha)y$. The following Hilbert sphere is the most important example of CAT(1) space in this paper.

\begin{definition}\normalfont
Let $(H,\langle\cdot,\cdot\rangle)$ be a real Hilbert space and $\|\cdot\|$ be its norm. The real Hilbert sphere $S_H$ is defined by $S_H:=\{ x\in H\ |\ \|x\|=1\}$ and we can define the metric function of $S_H$ by $d(x,y):=\arccos\langle x,y\rangle$.
\end{definition}

Then $S_H$ is an example of complete CAT(1) space, and thus a nonempty closed convex subset of $S_H$ is a complete CAT(1) space (see \cite{Bridson-Haefliger}). It is known that a hemisphere $\{z\in S_H\  |\ d(x,z)\leq d(y,z)\}$ and a subset $\{z\in S_H\ |\ \cos d(x,y)\cos d(y,z)\geq\cos d(x,z)\}$ are closed convex subsets of $S_H$.

\begin{theorem}[Kimura-Sat\^o \cite{KS2}]\label{pal}
Let $x,y,z$ be points in a {\rm CAT(1)} space such that $d(x,y)+d(y,z)+d(z,x)<2\pi$. Let $v:=tx\oplus(1-t)y$ for some $t\in[0,1]$. Then
\[
\cos d(v,z)\sin d(x,y)\geq \cos d(x,z)\sin(td(x,y)) + \cos d(y, z)\sin((1-t)d(x, y)).
\]
\end{theorem}

Let $X$ be a complete CAT(1) space such that $d(v,v')<\pi/2$ for all $v,v'\in X$, 
and let $C$ be a nonempty closed convex subset of $X$. 
Then for any $x \in X$, there exists a unique point $P_C x \in C$ such that
\[
 d(x,P_C x)=\inf_{y\in C}d(x,y).
\]
We call $P_C x$ the metric projection of $x \in X$. By definition, the metric projection $P_Cx$ is the nearest point of $C$ to a given $x \in X$. 

Let $X$ be a metric space and $\{x_n\}$ be a bounded sequence of $X$. 
The asymptotic center $AC(\{x_n\})$ of $\{x_n\}$ is defined by
\[
 AC(\{x_n\}):=\left\{ z \, \Big| \, 
 \limsup_{n \to \infty} d(z,x_n) = \inf_{x \in X}\limsup_{n\to\infty}d(x,x_n)\right\}.
\]
We say that $\{x_n\}$ is $\Delta$-convergent to a point $z \in X$ 
if for all subsequences $\{x_{n_i}\}$ of $\{x_n\}$, its asymptotic center consists only of $z$, that is, $AC(\{x_{n_i}\})=\{z\}$. 

Let $X$ be a metric space and $T$ be a mapping of $X$ into itself. Hereafter we denote by
\[
 F(T):=\{z \mid T z=z\}.
\]
the fixed point set of $T$. The mapping $T$ is called
\begin{itemize}
\item
nonexpansive
if $d(Tx,Ty)\leq d(x,y)$ for all $x,y\in X$,
\item
and $T$ is said to be quasinonexpansive if $d(Tx, p)\leq d(x, p)$ for all $x\in X$ and $p\in F(T)$. Using similar techniques to the case of Hilbert space, 
we can prove that $F(T)$ is a closed convex subset of $X$.
\item
We also say that $T$ is $\Delta$-demiclosed if for any $\Delta$-convergent sequence $\{x_n\}$ in $X$, its $\Delta$-limit belongs to $F(T)$ whenever $\lim_{n\to\infty}d(Tx_n,x_n)=0$.
\end{itemize}

The notation of $W$-mapping is originally proposed by Takahashi. We use the same notation in the setting of geodesic space as following: 

\begin{definition}[Takahashi-Shimoji \cite{Takahashi-Simoji}]\label{Wmapping}
Let $X$ be a geodesic metric space. Let $T_1,T_2,\ldots,T_r$ be a finite number of mappings of $X$ into itself and $\alpha_1,\alpha_2,\ldots,\alpha_r$ be real numbers such that $0\leq \alpha_i\leq 1$ for every $i=1,2,\ldots,r$. Then we define a mapping $W$ of $X$ into itself inductively as
\begin{align*}
U_1 &:=  \alpha_1T_1\oplus(1-\alpha_1)I,  \\
U_2 &:= \alpha_2T_2U_1\oplus(1-\alpha_2)I,  \\
 &\cdots& \\
U_r &:= \alpha_rT_rU_{r-1}\oplus(1-\alpha_r)I, \\
W &:= U_r.
\end{align*}
We obtained mapping $W:=U_r$ is called the $W$-mapping generated by $T_1,T_2,\ldots,T_r$ and $\alpha_1,\alpha_2,\ldots,\alpha_r$.
\end{definition}

The following lemmas are important for our main result.

\begin{lemma}[Kimura-Sat\^o \cite{Kimura-Sato}]\label{quasi nonexpansive}
Let $T$ be a nonexpansive mapping defined on a $\CAT(1)$ space. For any real number $\alpha\in[0,1]$, the mapping $\alpha T\oplus(1-\alpha)I$ is quasinonexpansive.
\end{lemma}

\begin{lemma}[Kimura-Sat\^o \cite{Kimura-Sato}]\label{delta-demiclosed}
Let $T$ be a nonexpansive mapping defined on a {\rm CAT(1)} space. For any real number $\alpha\in(0,1]$, the mapping $\alpha T\oplus(1-\alpha)I$ is $\Delta$-demiclosed.
\end{lemma}

\begin{lemma}[Esp\'\i nola-Fern\'andez-Le\'on \cite{Espinola-Fernandez-Leon}]\label{subseq}
Let $X$ be a complete {\rm CAT(1)} space, and $\{x_n\}$ be a sequence in $X$. If there exists $x\in X$ such that $\limsup_{n\to\infty}d(x_n,x)<\pi/2$, then $\{x_n\}$ has a $\Delta$-convergent subsequence.
\end{lemma}

\begin{lemma}[He-Fang-Lopez-Li \cite{He-Fang-Lopez-Li}]\label{conv}
Let $X$ be a complete $\CAT(1)$ space and $p\in X$. Let $\{x_n\}$ be a sequence in $X$. Supposed that $\{x_n\}$ satisfies $\limsup_{n\to\infty} d(x_n,p)<\pi/2$ and that $\{x_n\}$ is $\Delta$-convergent to
$x\in X$. Then $d(x, p)\leq \liminf_{n\to\infty}d(x_n,p)$.
\end{lemma}

\begin{lemma}[Kimura-Sat\^o \cite{KS2}]\label{weak}
Let $X$ be a complete {\rm CAT(1)} space such that $d(v,v')<\pi/2$ for every $v,v'\in X$ and $p\in X$. If a sequence $\{x_n\}$ in $X$ is $\Delta$-convergent to $x\in X$ and $\lim_{n\to\infty}d(x_n,p)=d(x,p)$, then $\{x_n\}$ is convergent to $x$.
\end{lemma}

\begin{lemma}[Kimura-Sat\^o \cite{KS3}]\label{weakclosed}
Let $\{x_n\}$ be a sequence in some closed convex subset of the Hilbert sphere $S_H$. If $\{x_n\}$ is $\Delta$-convergent to $x\in S_H$, then $x\in C$.
\end{lemma}

\begin{lemma}[Ezawa-Kimura \cite{Ezawa-Kimura}]\label{sin}
If
\[
\sin\delta\geq \sin(\alpha\delta)+\sin((1-\alpha)\delta)
\]
for some $\delta\in[0,\pi/2]$ and $\alpha\in(0,1)$, then $\delta=0$.
\end{lemma}

\begin{lemma}[Ezawa-Kimura \cite{Ezawa-Kimura}]\label{FixedPoints}
Let $X$ be a {\rm CAT(1)} space. Let $T_1,T_2,\ldots,T_r$ be quasinonexpansive mappings of X into itself such that $\bigcap_{i=1}^rF(T_i)\neq\emptyset$ and let $\alpha_1,\alpha_2,\ldots,\alpha_r$ be real numbers such that $0<\alpha_i<1$ for every $i=1,2,\ldots,r$. Let $W$ be the $W$-mappig of $X$ into itself generated by $T_1,T_2,\ldots,T_r$ and $\alpha_1,\alpha_2,\ldots,\alpha_r$. Then, $F(W)=\bigcap_{i=1}^rF(T_i)$.
\end{lemma}

\section{Main result}\label{sec:Main result}

\begin{theorem}\label{CQ}
Let $C$ be a closed convex subset in real Hilbert shpere $S_H$ such that $d(v,v')<\pi/2$ for every $v,v'\in C$. Let $\alpha_{n,1},a_{n,2},\ldots,\alpha_{n,r}$ be real numbers for $n\in\mathbb{N}$ such that $\alpha_{n,i}\in[a,1-a]$ for every $i=1,2,\ldots,r$ where $0<a<1/2$, and let $T_1,T_2,\ldots,T_r$ be a finite number of nonexpansive mappings of C into itself such that $F:=\bigcap_{i=1}^rF(T_i)\neq\emptyset$. Let $W_n$ be the W-mappings of $X$ into itself generated by $T_1,T_2,\ldots,T_r$ and $\alpha_{n,1},\alpha_{n,2},\ldots,\alpha_{n,r}$ for $n\in\mathbb{N}$. For a given point $x_1\in C$, let $\{x_n\}$ be a sequence in $C$ generated by
\begin{align*}
y_n &:=  W_nx_n, \\
C_n &:= \{z\in C\ |\ d(y_n,z)\leq d(x_n,z)\}, \\
Q_n &:= \{ z\in C\ |\ \cos d(x_1, x_n)\cos d(x_n,z)\geq\cos d(x_1,z)\}, \\
x_{n+1} &:= P_{C_n\cap Q_n}x_1
\end{align*}
for all $n\in\mathbb{N}$. Then $\{x_n\}$ is well defined and convergent to $P_Fx_1$.
\end{theorem}

\begin{proof}
First, we show that $\{x_n\}$ is well defined, that is, we show that $C_n\cap Q_n$ is nonempty closed convex subset. By the definition of $C_n$ and $Q_n$, $C_n$ and $Q_n$ are closed subsets. Since $C_n$ and $Q_n$ are hemisphere, $C_n$ and $Q_n$ are convex subsets. Thus $C_n\cap Q_n$ is closed convex subset in $C$. In order to show that it is nonempty, we show $F\subset C_n\cap Q_n$ by induction for $n\in\mathbb{N}$. Since $C_1=Q_1=C$, we have $F\subset C_1\cap Q_1$, and $C_1\cap Q_1$ is nonempty closed convex subset. We assume the induction hypothesis that $F\subset C_k\cap Q_k$ and show $F\subset C_{k+1}\cap Q_{k+1}$. For all $z\in F$, by Lemma \ref{FixedPoints}, $d(W_{k+1}x_{k+1},z)\leq d(x_{k+1},z)$ and thus $z\in C_{k+1}$. By induction hypothesis, $z\in C_k\cap Q_k$, and therefore, for any $t\in[0,1], tz\oplus(1-t)x_{k+1}=tz\oplus(1-t)P_{C_k\cap Q_k}x_1\in C_k\cap Q_k$. Then we have
\begin{align*}
& 2\cos d(x_1,x_{k+1})\cos\left(\left(1-\dfrac{t}{2}\right)d(x_{k+1},z)\right)\sin\left(\dfrac{t}{2}d(x_{k+1},z)\right) \\
&= \cos d(x_1,x_{k+1})(\sin d(x_{k+1},z)-\sin((1-t)d(x_{k+1},z))) \\
&= \cos d(x_1,P_{C_k\cap Q_k}x_1)\sin d(x_{k+1},z)-\cos d(x_1,x_{k+1})\sin((1-t)d(x_{k+1},z)) \\
&\geq \cos d(x_1,tz\oplus(1-t)x_{k+1})\sin d(x_{k+1},z)-\cos d(x_1,x_{k+1})\sin((1-t)d(x_{k+1},z)). \\
\end{align*}
By Theorem \ref{pal}, the last expression is estimated as 
\begin{align*}
&\geq \cos d(x_1,z)\sin(td(z,x_{k+1}))+\cos d(x_1,x_{k+1})\sin((1-t)d(z,x_{k+1})) \\
& -\cos d(x_1,x_{k+1})\sin((1-t)d(z,x_{k+1})) \\
&= \cos d(x_1,z)\sin(td(z,x_{k+1})) \\
&= 2 \cos d(x_1,z)\sin\left(\dfrac{t}{2}d(z,x_{k+1})\right)\cos\left(\dfrac{t}{2}d(z,x_{k+1})\right).
\end{align*}
If $z=x_{k+1}$, then it is obvious that $z\in Q_{k+1}$ by definition of $Q_k$. So, we assume that $z\neq x_{k+1}$. Dividing above by $2\sin(td(z,x_{k+1})/2)$ and letting $t\to 0$, we have
\[
\cos d(x_1,x_{k+1})\cos d(x_k,z)\geq\cos d(x_1,z)
\]
and thus $z\in Q_{k+1}$. From the above, we get $z\in C_{k+1}\cap Q_{k+1}$ and $F\subset C_{k+1}\cap Q_{k+1}$. Therefore, $C_n\cap Q_n$ is nonempty closed convex subset and $\{x_n\}$ is well defined.

Next, we show that $\lim_{n\to\infty}d(T_ix_n,x_n)=0$ for all $i=1,2,\ldots,r$ to get our result. By the definition of metric projection, for all $n\in\mathbb{N}$, and letting $n\to\infty$,we obtain
\[
d(x_1,x_n)=d(x_1,P_{C_{n-1}\cap Q_{n-1}}x_1)\leq d(x_1,P_Fx_1)<\dfrac{\pi}{2}
\]
for all $n\in\mathbb{N}\setminus\{1\}$ and hence $\sup_{n\in\mathbb{N}}d(x_1,x_n)\leq d(x_1,P_Fx_1)<\pi/2$. By the definition of $Q_n$, we have
\[
d(x_1,x_n)=d(x_1,P_{Q_n}x_1)\leq d(x_1,P_{C_n\cap Q_n}x_1)=d(x_1,x_{n+1}).
\]
Then, $\{\cos d(x_1,x_n)\}$ is a monotonically non-increasing sequence of real numbers. Then we can put
\[
a:=\lim_{n\to\infty}\cos d(x_1,x_n)>\cos\dfrac{\pi}{2}=0
\]
By the definition $x_{n+1}=P_{C_n\cap Q_n}x_1\in Q_n$, we have
\[
\cos d(x_1,x_n)\cos d(x_n,x_{n+1})\geq\cos d(x_1,x_{n+1}).
\]
for all $n\in\mathbb{N}$ hold. we have
\[
a\liminf_{n\to\infty}\cos d(x_n,x_{n+1})\geq a.
\]
It follow that
\[
1\geq \cos\left(\limsup_{n\to\infty}d(x_n,x_{n+1})\right)=\liminf_{n\to\infty}\cos d(x_n,x_{n+1})\geq 1,
\]
and we have that $\limsup_{n\to\infty}d(x_n,x_{n+1})=0$. Hence $\lim_{n\to\infty}d(x_n,x_{n+1})=0$. By the definition $C_n$ and $x_{n+1}$, it follows that $x_{n+1}\in C_n$ and then
\[
d(W_nx_n,x_{n+1})\leq d(x_n,x_{n+1}).
\]
We have that
\begin{align*}
0 &\leq \alpha_{n,r}d(T_rU_{n,r-1}x_n,x_n) \\
&= d(\alpha_{n,r}T_rU_{n,r-1}x_n\oplus(1-\alpha_{n,r})x_n,x_n) \\
&= d(W_nx_n,x_n) \\
&\leq d(W_nx_n,x_{n+1})+d(x_{n+1},x_n) \\
&\leq d(x_n,x_{n+1})+d(x_{n+1},x_n) \\
&= 2d(x_n,x_{n+1})\to 0\ (n\to\infty).
\end{align*}
Since $\inf_{n\in\mathbb{N}}\alpha_{n,r}>0$, we get
\[
\lim_{n\to\infty}d(T_r,U_{n,r-1}x_n,x_n)=0.
\]
Next we show
\[
\lim_{n\to\infty}d(T_{r-1}U_{n,r-2}x_n,x_n)=0.
\]
Since $P_{C_n\cap Q_n}$ is quasinonexpansive and $z\in F\subset C_n\cap Q_n$, we have that
\[
d(x_{n+1},z)=d(P_{C_n\cap Q_n}x_1,z)\leq d(x_1,z)<\dfrac{\pi}{2}.
\]
Thus we have that
\[
\inf_{n\in\mathbb{N}}\cos d(x_n,z)=\cos\left(\sup_{n\in\mathbb{N}} d(x_n,z)\right)>\cos\dfrac{\pi}{2}=0.
\]
Put $\varepsilon_n:=d(T_rU_{n,r-1}x_n,x_n)$ and $\delta_n:=d(T_{r-1}U_{n,r-2}x_n,x_n)$. Then, we have
\begin{align*}
& \quad\cos d(x_n,z)\sin d(T_{r-1}U_{n,r-2}x_n,x_n) \\
&\geq \cos(d(x_n,T_rU_{n,r-1}x_n)+d(T_rU_{n,r-1}x_n,z))\sin d(T_{r-1}U_{n,r-2}x_n,x_n) \\
&= \cos(\varepsilon_n+d(T_rU_{n,r-1}x_n,z))\sin d(T_{r-1}U_{n,r-2}x_n,x_n) \\
&= \{\cos\varepsilon_n\cos d(T_rU_{n,r-1}x_n,z)-\sin\varepsilon_n\sin d(T_rU_{n,r-1}x_n,z)\}\sin d(T_{r-1}U_{n,r-2}x_n,x_n) \\
&= \cos\varepsilon_n\cos d(T_rU_{n,r-1}x_n,z)\sin d(T_{r-1}U_{n,r-2}x_n,x_n) \\
& \quad-\sin\varepsilon_n\sin d(T_rU_{n,r-1}x_n,z)\sin d(T_{r-1}U_{n,r-2}x_n,x_n) \\
\end{align*}
Since $T_r$ is nonexpansive,
\[
\geq \cos\varepsilon_n\cos d(U_{n,r-1}x_n,z)\sin d(T_{r-1}U_{n,r-2}x_n,x_n)-\sin\varepsilon_n\sin d(T_rU_{n,r-1}x_n,z)\sin\delta_n \\
\]
By Theorem \ref{pal}
\begin{align*}
&\geq \cos\varepsilon_n\{\cos d(T_{r-1}U_{n,r-2}x_n,z)\sin(\alpha_{n,r-1}d(T_{r-1}U_{n,r-2}x_n,x_n)) \\
& \quad+\cos d(x_n,z)\sin((1-\alpha_{n,r-1})d(T_{r-1}U_{n,r-2}x_n,x_n))\} -\sin\varepsilon_n\sin d(T_rU_{n,r-1}x_n,z)\sin\delta_n \\
\end{align*}
By Lemma \ref{quasi nonexpansive}, $T_{r-1}U_{n,r-2}$ is quasinonexpansive, so
\begin{align*}
&\geq \cos\varepsilon_n\{\cos d(x_n,z)\sin(\alpha_{n,r-1}d(T_{r-1}U_{n,r-2}x_n,x_n)) \\
& \quad+\cos d(x_n,z)\sin((1-\alpha_{n,r-1})d(T_{r-1}U_{n,r-2}x_n,x_n))\}-\sin\varepsilon_n\sin d(T_rU_{n,r-1}x_n,z)\sin \delta_n \\
&= \cos\varepsilon_n\cos d(x_n,z)\{\sin(\alpha_{n,r-1}d(T_{r-1}U_{n,r-2}x_n,x_n)) +\sin((1-\alpha_{n,r-1})d(T_{r-1}U_{n,r-2}x_n,x_n))\} \\
& \quad-\sin\varepsilon_n\sin d(T_rU_{n,r-1}x_n,z)\sin\delta_n
\end{align*}
and hence
\begin{align*}
\cos d(x_n,z)\sin\delta_n &\geq \cos\varepsilon_n\cos d(x_n,z)\{\sin(\alpha_{n,r-1}\delta_n)+\sin((1-\alpha_{n,r-1})\delta_n)\} \\
& \quad-\sin\varepsilon_n\sin d(T_rU_{n,r-1}x_n,z)\sin\delta_n.
\end{align*}
Then dividing above by $\cos d(x_n,z)$, we have
\[
\sin\delta_n \geq \cos\varepsilon_n\{\sin(\alpha_{n,r-1}\delta_n)+\sin((1-\alpha_{n,r-1})\delta_n)\} -\dfrac{\sin\varepsilon_n\sin d(T_rU_{n,r-1}x_n,z)\sin\delta_n}{\cos d(x_n,z)}
\]
Let $\{\delta_{n_i}\}$ be a convergent subsequence whose limit is $\delta\in [0,\pi/2]$. There exists subsequence $\{\alpha_{n_{i_j},r-1}\}$ of $\{\alpha_{n_i,r-1}\}$ and $\alpha\in(0,1)$ such that $\alpha_{n_{i_j},r-1}\to\alpha$ as $j\to\infty$. Then since $\varepsilon_{n_{i_j}}\to 0$ as $j\to\infty$, we get
\[
\sin\delta\geq \sin(\alpha\delta)+\sin((1-\alpha)\delta).
\]
By Lemma \ref{sin} $\delta=0$. Therefore $\{\delta_n\}$ converges to $0$, that is,
\[
\lim_{n\to\infty}d(T_{r-1}U_{n,r-1}x_n,x_n)=0.
\]
Using a similar calculation inductively, we have
\[
\lim_{n\to\infty}d(T_iU_{n,i-1}x_n,x_n)=0
\]
for all $i=1,2,\ldots,r$. Since
\begin{align*}
d(T_ix_n,x_n) &\leq d(T_ix_n,T_iU_{n,i-1}x_n)+d(T_iU_{n,i-1}x_n,x_n) \\
&\leq d(x_n,U_{n,i-1}x_n)+d(T_iU_{n,i-1}x_n,x_n) \\
&= d(x_n,\alpha_{n,i-1}T_{i-1}U_{n,i-2}x_n\oplus(1-\alpha_{n,i-1})x_n)+d(T_iU_{n,i-1}x_n,x_n) \\
&= \alpha_{n,i-1}d(T_{i-1}U_{n,i-2}x_n,x_n)+d(T_iU_{n,i-1}x_x,x_n) \\
&\to 0 \ (n\to\infty),
\end{align*}
we obtain $\lim_{n\to\infty}d(T_ix_n,x_n)=0$ for all $i=1,2,\ldots,r$. Let $\{x_{n_i}\}$ be an arbitrary subsequence of $\{x_n\}$. By the inequality $\sup_{j\in\mathbb{N}}d(x_1,x_{n_j})\leq\sup_{n\in\mathbb{N}}d(x_1,x_n)<\pi/2$ and Lemma \ref{subseq}, there exists a subsequence $\{x_{n_{i_j}}\}$ of $\{x_{n_i}\}$ and $w_\infty$ such that $\{x_{n_{i_j}}\}$ is $\Delta$-convergent to $w_\infty$. For the sake of simplicity, we put $w_j:=x_{n_{i_j}}$. Then we can show $w_\infty\in \bigcap_{i=1}^rF(T_i)$. For all $i=1,2,\ldots,r$,
\begin{align*}
\limsup_{j\to\infty} d(w_j,T_iw_\infty) &\leq \limsup_{j\to\infty}(d(w_j,T_iw_j)+d(T_iw_j,T_iw_\infty)) \\
&\leq \limsup_{j\to\infty}(d(w_j,T_iw_j)+d(w_j,w_\infty)) \\
&= \limsup_{j\to\infty}d(w_j,w_\infty).
\end{align*}
By the definition of $\Delta$-convergence, we get $T_iw_\infty=w_\infty$. Hence $w_\infty\in\bigcap_{i=1}^rF(T_i)$. Since
\[
w_j=x_{n_{i_j}}=P_{C_{n_{i_j}-1}\cap Q_{n_{i_j}-1}}x_1\in C_{n_{i_j}-1}\cap Q_{n_{i_j}-1},
\]
$F\subset C_{n_{i_j}-1}\cap Q_{n_{i_j}-1}$ and Lemma \ref{conv}
\[
d(x_1,P_Fx_1)\leq d(x_1,w_\infty)\leq\lim_{j\to\infty}d(x_1,w_j)\leq d(x_1,P_Fx_1).
\]
Thus we get $\lim_{j\to\infty}d(x_1,w_j)=d(x_1,w_\infty)$ and by Lemma \ref{subseq}, $\{w_j\}$ is convergent to $w_\infty$. On the other hand, we get $d(x_1,P_Fx_1)=d(x_1,w_\infty)$ and by definition of the metric projection $P_F$, we get $w_\infty=P_Fx_1$. From the above, any subsequence $\{x_{n_i}\}$ of $\{ x_n\}$ has a subsequence $\{x_{n_{i_j}}\}$ such that $\{x_{n_{i_j}}\}$ is convergent to $P_Fx_1$. Therefore $\{x_n\}$ is convergent to $P_Fx_1$.
\end{proof}

Next, we consider the shrinking projection method using $W$-mapping on a complete CAT(1) space.

\begin{theorem}\label{Shrinking}
Let $C$ be a closed convex subset in real Hilbert shpere $S_H$ such that $d(v,v')<\pi/2$ for every $v,v'\in C$. Let $\alpha_{n,1},a_{n,2},\ldots,\alpha_{n,r}$ be real numbers for $n\in\mathbb{N}$ such that $\alpha_{n,i}\in[a,1-a]$ for every $i=1,2,\ldots,r$ where $0<a<1/2$, and let $T_1,T_2,\ldots,T_r$ be a finite number of nonexpansive mappings of C into itself such that $F:=\bigcap_{i=1}^rF(T_i)\neq\emptyset$. Let $W_n$ be the W-mappings of $X$ into itself generated by $T_1,T_2,\ldots,T_r$ and $\alpha_{n,1},\alpha_{n,2},\ldots,\alpha_{n,r}$ for $n\in\mathbb{N}$. For a given point $x_1\in C$, let $\{x_n\}$ be a sequence in $C$ generated by
\begin{align*}
x_1 &:=x\in C, \\
y_n &:=W_nx_n, \\
C_n &:=\{z\in C\ |\ d(y_n,z)\leq d(x_n,z)\}\cap C_{n-1}, \\
x_{n+1} &:=P_{C_n}x_1,
\end{align*}
for all $n\in\mathbb{N}$. Then $\{x_n\}$ is well defined and convergent to $P_Fx_1$.
\end{theorem}

\begin{proof}
First, we show that $\{x_n\}$ is well defined, that is, we show that $C_n$ is nonempty closed convex subset of $C$. By the definition of $C_n$, we have that $C_n$ is a closed convex subset of $C$. So, we shall show $C_n$ is nonempty. Since $T_1,T_2,\ldots,T_r$ are nonexpansive, by Lemma \ref{quasi nonexpansive}, $W_n$ is quasinonexpansive. Then by Lemma \ref{FixedPoints}, we have that $F\subset \{z\in C\ |\ d(y_n,z)\leq d(x_n,z)\}$. Then by induction, we have that $F\subset C_n$ for all $n\in\mathbb{N}$, that is, $C_n$ is nonempty for all $n\in\mathbb{N}$. Therefore, $\{x_n\}$ is well defined.

Next, we show that $\{x_n\}$ is convergent to $P_Fx_1$. Let $\{x_{n_i}\}$ be an arbitrary subsequence of $\{x_n\}$. By the inequality $\sup_{j\in\mathbb{N}}d(x_1,x_{n_j})\leq\sup_{n\in\mathbb{N}}d(x_1,x_n)<\pi/2$, there exists a subsequence $\{x_{n_{i_j}}\}$ of $\{x_{n_i}\}$ and $w_\infty$ such that $\{x_{n_{i_j}}\}$ is $\Delta$-convergent to $w_\infty$. For the sake of simplicity, we put $w_j:=x_{n_{i_j}}$. For all $k\in\mathbb{N}$, there exists $j_0\in\mathbb{N}$ such that, for every $j\geq j_0, w_j\in C_k$. By Lemma \ref{weakclosed}, we get $w_\infty\in C_k$. Thus we have that $w_\infty\in\bigcap_{k=1}^\infty C_k$, and by Lemma \ref{conv}
\[
d(x_1,P_{\bigcap_{k=1}^\infty C_k}x_1)\leq d(x_1,w_\infty)\leq\lim_{j\to\infty}d(x_1,w_j)=\lim_{j\to\infty}d(x_1,P_{C_{n_{i_j}}}x_1)\leq d(x_1,P_{\bigcap_{k=1}^\infty C_k}x_1).
\]
Thus we get $d(x_1,w_\infty)=\lim_{j\to\infty}d(x_1,w_j)$ and by Lemma \ref{subseq}, $\{w_j\}$ convergent to $w_\infty$. On the other hand, we get $d(x_1,P_{\bigcap_{k=1}^\infty C_k}x_1)=d(x_1,w_\infty)$. By the definition of the metric projection $P_{\bigcap_{k=1}^\infty C_k}$, we get $w_\infty=P_{\bigcap_{k=1}^\infty C_k}x_1$. From the above, for any subsequence $\{x_{n_i}\}$ of $\{ x_n\}$ has a subsequence $\{x_{n_{i_j}}\}$ such that $\{x_{n_{i_j}}\}$ convergent to $P_{\bigcap_{k=1}^\infty C_k}x_1$. Since $P_{\bigcap_{k=1}^\infty C_k}x_1\in \bigcap_{k=1}^\infty C_k\subset C_n$ for all $n\in\mathbb{N}$ and the definition of $C_n$, $d(W_nx_n,P_{\bigcap_{k=1}^\infty C_k}x_1)\leq d(x_n, P_{\bigcap_{k=1}^\infty C_k}x_1)$. Thus
\[
\limsup_{n\to\infty}d(W_nx_n,P_{\bigcap_{k=1}^\infty C_k}x_1)\leq\lim_{n\to\infty}d(x_n, P_{\bigcap_{k=1}^\infty C_k}x_1)=0
\]
holds. Hence $\{W_nx_n\}$ is convergent to $P_{\bigcap_{k=1}^\infty C_k}x_1$. Therefore, $\lim_{n\to\infty}d(W_nx_n,x_n)=0$. Then, as in the proof of Theorem \ref{CQ}, we have that $\lim_{n\to\infty}d(T_ix_n,x_n)=0$ for all $i=1,2,\ldots,r$. Thus $P_{\bigcap_{k=1}^\infty C_k}x_1\in F(T_i)$ for all $r=1,2,\ldots,r$. Thus $P_{\bigcap_{k=1}^\infty C_k}x_1\in\bigcap_{i=1}^rF(T_i)$. Since $F\subset \bigcap_{k=1}^\infty C_k$ and $P_{\bigcap_{k=1}^\infty C_k}x_1\in F$,
\[
d(x_1,P_{\bigcap_{k=1}^\infty C_k}x_1)\leq d(x_1,P_Fx_1)\leq d(x_1,P_{\bigcap_{k=1}^\infty C_k}x_1).
\]
By the definition of the metric projection, we have $P_{\bigcap_{k=1}^\infty C_k}x_1=P_Fx_1$, that is, $\{x_n\}$ is convergent to $P_Fx_1$, and we finish the proof.
\end{proof}

\section*{Acknowledgment}
The author would like to thank Professor Yasunori Kimura for helpful conversations and valuable comments.

\end{document}